\documentclass{birkjour}
\usepackage{amsmath}
\usepackage{amssymb}
\usepackage{amsthm}
\usepackage{eqlist}
\usepackage{enumerate}
\usepackage[mathscr]{eucal}
\usepackage{hyperref}
\hypersetup{
     colorlinks   = true,
     citecolor    = blue
}
\hypersetup{linkcolor=blue}

 \newtheorem{theorem}{Theorem}[section]
 
 \newtheorem{lemma}[theorem]{Lemma}
 \newtheorem{proposition}[theorem]{Proposition}
 \theoremstyle{definition}
 \newtheorem{definition}[theorem]{Definition}
 \theoremstyle{remark}
 \newtheorem{rem}[theorem]{Remark}
 \newtheorem{example}[theorem]{Example}
 
 \newtheorem*{acknowledgement}{\textnormal{\textbf{Acknowledgement}}}
 \numberwithin{equation}{section}

\newcommand{\RR}{\mathbb{R}}

\begin{document}
\title[$\mathcal{P}\mathcal{R}$-anti-slant warped product submanifold]
{$\mathcal{P}\mathcal{R}$-anti-slant warped product submanifold of a nearly paracosymplectic manifold}
\author[S. K. Srivastava, A. Sharma] {S. K. Srivastava, A. Sharma}
\address{Department of Mathematics, Central University of Himachal Pradesh , Dharamshala-176215, Himachal Pradesh
	       INDIA}
\email{sachin@cuhimachal.ac.in, anilsharma3091991@gmail.com}
\author{S. K. Tiwari}
\address{Department of Applied science,  Ajay Kumar Garg Engg. College, Ghaziabad-201009,  Uttar Pradesh
	       INDIA}
\email{drstiwari.ims@gmail.com}
\begin{abstract}
In this paper, we study $\mathcal{P}\mathcal{R}$-anti-slant warped product submanifold of a nearly paracosymplectic manifold $\widetilde{M}$. The necessary and sufficient condition is obtained for the distributions allied to the characterization of a $\mathcal{P}\mathcal{R}$-anti-slant submanifold being integrable and totally geodesic foliation. In addition, we have defined $\mathcal{P}\mathcal{R}$-anti-slant warped product submanifold of $\widetilde{M}$ and gave some illustrations. Finally, we extracted the constraints for a submanifold of $\widetilde{M}$ to be a $\mathcal{P}\mathcal{R}$-anti-slant warped product of the form $F\times_{f}N_{\lambda}$.
\end{abstract}
\thanks{S. K. Srivastava: partially supported through the UGC-BSR Start-Up-Grant vide their letter no. F.30-29/2014(BSR). A. Sharma: supported by Central University of Himachal Pradesh through the Research fellowship for Ph.D.} 
\maketitle
\section{Introduction}\label{intro}
\noindent The warped product submanifolds of a pseudo-Reimannian manifold served as a fruitful platform in exploring, identifying and to solve problems in mathematical physics, especially in different models of spacetime, black holes, Ricci flow and Hamiltonian spaces (c.f.,\cite{HM,JC,PVK,HT}). In \cite{OB}, Bishop and O'Neill initiated the theory of warped product submanifold as a generalization of pseudo-Reimannian product manifolds. The study has attained momentum when Chen \cite{BY} introduced the geometric depiction of warped product CR-submanifolds in Kahlerian manifold $\widetilde{N}$ through differential point of view and proved the non-existence of proper warped product CR-submanifolds in the form $N_{\bot}\times {}_fN_{T}$ such that $N_{T}$ is a holomorphic submanifold and $N_{\bot}$ is a totally real submanifold of $\widetilde{N}$.

Apart from warped product submanifolds, there is a major generalized class of both holomorphic and totally real submanifolds, called slant submanifolds. Chen begun the concept of slant submanifolds in complex geometry \cite{BC}. Later on, slant and semi-slant submanifolds in contact Riemannian geometry are studied by Lotta and Cabrerizo, respectively \cite{LC, LC1, AL}. This geometric setting may not found suitable in mathematical physics particularly, in the theory of space time and black hole, where the metric is not necessarily positive definite. Thus, the geometry of slant submanifolds with indefinite metric became a topic of investigation, and Sahin gave the notion of slant lightlike submanifolds of indefinite Hermitian manifolds \cite{BS}. Since then several geometers have contributed many important characterization to the warped product slant submanifolds in almost contact, complex and lorentzian manifolds (c.f.,\cite{PA, BYC, BS1, UDS}). Recently, Chen-Munteanu, brought our attention to the geometry of $\mathcal{P}\mathcal{R}$-warped products in para-K\"{a}hler manifolds \cite{CPR}. Motivated by the work of \cite{CPR}, the authors have studied $\mathcal{P}\mathcal{R}$-warped product in paracontact manifold which can be viewed as the counterpart of para- K\"{a}hler manifold \cite{SA}. 

This paper is organized as follows. In Sect.\ref{pre}, the basic informations about almost paracontact metric manifolds, nearly paracosymplectic manifold and slant submanifold are given. Sect.\ref{pranti-sub}, concerned with $\mathcal{P}\mathcal{R}$-anti-slant submanifold of a nearly paracosymplectic manifold $\widetilde{M}$. The integrability and totally geodesic foliation conditions for the distributions involved with the definition are drawn. In Sect.\ref{prwanti} we define $\mathcal{P}\mathcal{R}$-anti-slant warped product submanifold $M$ and investigate the existence and nonexistence results for such submanifolds  of $\widetilde{M}$. Further, we give some examples of a $\mathcal{P}\mathcal{R}$-anti-slant warped product submanifold $F\times_{f}N_{\lambda}$, where $F$ is anti-invariant and $N_{\lambda}$ is proper slant submanifold of a nearly paracosymplectic manifold $\widetilde{M}$. Finally, we derive the necessary and sufficient condition for a submanifold of $\widetilde{M}$ to be a $\mathcal{P}\mathcal{R}$-anti-slant warped product of the form $F\times_{f}N_{\lambda}$.
\section{Preliminaries}\label{pre}
\subsection{Almost paracontact metric manifolds}
A $(2n + 1)$-dimensional smooth manifold $\widetilde{M}$ is said to have  an  almost paracontact structure $(\varphi ,\xi ,\eta )$, if there exist on $\widetilde{M}$ a tensor field $\varphi$ of type $(1, 1)$, a vector field $\xi$, and a $1$-form $\eta$  satisfying 
\begin{align}
  &\varphi ^{2} =Id-\eta \otimes \xi ,\quad \eta (\xi )=1\label{phieta}
\end{align}
where $Id$ is the identity transformation and the tensor field $\varphi$ induces an almost paracomplex structure on the distribution $D = \ker(\eta$), that the eigen distributions $D^{\pm}$ corresponding to the eigenvalues $\pm1$ have equal dimensions $n$.
From the equation \eqref{phieta}, it can be easily deduced that 
\begin{align}\label{phixi}
\varphi\xi = 0, \quad \eta\circ\varphi = 0 \quad {\rm and \quad rank}(\varphi)=2n.
\end{align}
If the manifold $\widetilde{M}$ has an almost paracontact structure $(\varphi ,\xi ,\eta )$ then we can find a non-degenerate pseudo-Riemannian metric $g$ on  $\widetilde{M}$ such that 
\begin{align}
g\left(X,Y\right)=-g(\varphi X,\varphi Y)+\eta (X)\eta (Y),\label{metric}
\end{align}
where signature of $g$ is necessarily $(n+1,\,n)$ for any vector fields $X$ and $Y$; then the quadruple $(\varphi,\xi,\eta,g)$ is called an almost paracontact metric structure and the manifold $\widetilde{M}$ equipped with paracontact metric structure is called an almost paracontact metric manifold \cite{SK,ZA}. With respect to $g$, $\eta$ is metrically dual to $\xi$, that is
\begin{align}\label{geta}
g(X,\xi)=\eta(X). 
\end{align}
With the consequences of Eqs. \eqref{phieta}, \eqref{phixi} and \eqref{metric} we deduce
\begin{align}\label{antisymphi}
g(\varphi X,Y)=-g(X,\varphi Y),
\end{align}
for any $X,Y \in \Gamma(T\widetilde{M})$. Here $\Gamma(T\widetilde{M})$ is the tangent bundle of $\widetilde{M}$. 
Finally, the fundamental $2$-form $\Phi$ on $\widetilde{M}$ is given by
 \begin{align}\label{PHI} 
 g(X,\varphi Y)=\Phi(X,Y). 
 \end{align}
\noindent
\begin{definition} 
For all $X,Y \in \Gamma(T\widetilde{M})$ an almost paracontact metric manifold $(\widetilde{M},\varphi,\xi,\eta,g)$ is called (\cite{PD, ES, IK}) 
 \begin{itemize}
 \item[$\bullet$] nearly para Sasakian if 
  $$(\widetilde{\nabla}_X \varphi)Y+(\widetilde{\nabla}_Y \varphi)X=2g(X,Y)\xi+(\eta(X)Y+\eta(Y)X)$$
  where, $\widetilde{\nabla }$ is  Levi-Civita connection on $\widetilde{M}$.
 \item[$\bullet$] paracosymplectic if the forms $\eta$ and $\varphi$ are parallel with respect to the Levi-Civita connection $\widetilde{\nabla }$ on $\widetilde{M}$, i.e., 
$$\widetilde{\nabla }\eta=0\quad{\rm and}\quad \widetilde{\nabla }\varphi=0.$$
\item[$\bullet$] nearly paracosymplectic if $\varphi$ is killing, i.e.,
\begin{align}\label{npcos}
(\widetilde{\nabla}_X \varphi)Y+(\widetilde{\nabla}_Y\varphi)X=0
\end{align} 
or equivalently,
\begin{align}\label{nabxx}
(\widetilde{\nabla}_X \varphi)X=0.
\end{align}
\end{itemize}
\end{definition}
\noindent Now, for $X, Y \in \Gamma(T\widetilde M)$, $N \in \Gamma(T\widetilde M)^{\bot}$ and by the property of Levi-Civita connection $\widetilde{\nabla }$ on $\widetilde{M}$, we have
\begin{align}\label{cplevi}
g(\widetilde{\nabla}_{X}Y,N)=-g(\widetilde{\nabla }_{X}N,Y)
\end{align}
here, $\Gamma(T\widetilde M)^{\bot}$ denotes the set of vector fields normal to $\widetilde{M}$.
Following Blair \cite{BD} we prove: 
\begin{proposition}\label{killing}
On a nearly paracosymplectic manifold the vector field $\xi$ is killing.
\end{proposition}
\begin{proof}
Clearly by using Eq. \eqref{metric} we obtain that $g(\widetilde\nabla_{\xi}\xi,\xi)=0$, so we can consider that $X$ (any non-zero tangent vector) is orthogonal to $\xi$.
Now, by using the fact that manifold is nearly paracosymplectic we get
\begin{align*}
g(\widetilde\nabla_{X}\xi, X)=&-g(\varphi\widetilde\nabla_{X}\xi, \varphi X)=g((\widetilde\nabla_{X}\varphi)\xi,\varphi X)=0.
\end{align*}
This completes the proof.
\end{proof}

\subsection{Geometry of slant submanifolds}
\noindent Let $M$ be a submanifold immersed in a $(2n + 1)$-dimensional almost paracontact manifold $\widetilde{M}$; we denote by the same symbol $g$ the induced non-degenerate metric on $M$. If $\Gamma(TM)$ denotes the tangent bundle of submanifold $M$ and $\Gamma(TM^{\bot })$ the set of vector fields normal to $M$ then Gauss and Weingarten formulas are given by respectively
\begin{align}
\widetilde{\nabla}_X  Y &= \nabla _{X} Y+h(X,Y), \label{gauss}\\
\widetilde{\nabla}_X  \zeta &=-A_{\zeta} X+\nabla _{X}^{\bot } \zeta. \label{wngrtn}
\end{align}
for any $X,Y \in \Gamma(TM)$ and $\zeta \in \Gamma(TM^{\bot })$, where $\nabla$ is the induced connection, $\nabla ^{\bot }$ is the normal connection on the normal bundle  $\Gamma(TM^{\bot })$, $h$ is the second fundamental form, and the shape operator $A_{\zeta}$ associated with the normal section $\zeta$ is given by
\begin{align}
 \label{shp2form} g\left(A_{\zeta} X,Y\right)=g\left(h(X,Y),\zeta\right).
\end{align}
 If we write, for all $X \in \Gamma(TM)$ and $\zeta \in \Gamma(TM^{\bot })$ that
\begin{align} 
\varphi X&=tX+nX,\label{phix}\\
\varphi \zeta&=t'\zeta+n'\zeta, \label{phin}
\end{align}
where $tX$ (resp., $nX$) is tangential (resp., normal) part of $\varphi X$ and $t'\zeta$ (resp., $n'\zeta$) is tangential (resp., normal) part of $\varphi \zeta$.  Then the submanifold $M$ is said to be {\it invariant} if $n$ is identically zero and {\it anti-invariant} if $t$ is identically zero. From Eqs.\eqref{antisymphi} and \eqref{phix}, we obtain that 
\begin{align} \label{antixty}
g(X,tY)=-g(tX,Y). 
\end{align}
The mean curvature vector $H$ of $M$ is given by $H = \frac{1}{n}{\rm trace}\, h$. A  submanifold $M$ is said to be \cite{CPR} 
\begin{itemize}
\item[$\bullet$] {\it totally geodesic} if its second fundamental form vanishes identically. 
\item[$\bullet$] {\it umbilical} in the direction of a normal vector field $\zeta$ on $M$, if $A_{\zeta} = \delta Id$, for certain function $\delta$ on $M$, here such $\zeta$ is called a umbilical section. 
\item[$\bullet$] {\it totally umbilical} if $M$ is umbilical with respect to every (local) normal vector field. 
\end{itemize}
For all $ X, Y  \in \Gamma(T\widetilde{M})$, the covariant derivative of tensor field $\varphi$ is  defined as
\begin{align}\label{covphi}
(\widetilde{\nabla}_X{\varphi})Y=\widetilde{\nabla}_X{\varphi}Y-\varphi\widetilde\nabla_X{Y}.
\end{align}
Let $\mathcal{T}_{X}Y$ be the tangential and $\mathcal{N}_{X}Y$ be the normal part of $(\widetilde\nabla_X{\varphi})Y$ then
\begin{align}\label{phitn}
(\widetilde{\nabla}_X{\varphi})Y=\mathcal{T}_{X}Y+\mathcal{N}_{X}Y, \, \, \forall X , Y  \in \Gamma(TM). 
\end{align}
For later use we can verify the property of $\mathcal{T}$ and $\mathcal{N}$ given by,
\begin{align} \label{P7}
 g(\mathcal{T}_{X}Y, W)=-g(Y, \mathcal{T}_{X}W) ,
\end{align}
$\forall \, X, Y, W \in \Gamma(TM)$ and $N \in \Gamma(TM)^{\bot}$.
On a submanifold $M$ of a nearly parcosymplectic manifold, by Eqs.\eqref{npcos} and \eqref{phitn}, we have
\begin{align}
 \mathcal{T}_{X}Y+\mathcal{T}_{Y}X&=0 ,\label{tantan}\\
\mathcal{N}_{X}Y+\mathcal{N}_{Y}X&=0 ,\label{nornor}
\end{align}
 for any $ X, Y \in \Gamma(TM)$.
 
\noindent Following the notion of slant submanifold in \cite{PA, LC1}. We give the following definition:
\begin{definition}\label{slantdfn}
Let $M$ be an isometrically immersed submanifold of an almost paracontact manifold $\widetilde{M}(\varphi,\xi,\eta,g)$ and $\mathfrak{D_{\lambda}}$ be the non-degenerate distribution on $M$. Then $M$ is said to be \textit{slant} submanifold of $\widetilde{M}$ if there exists a constant $\lambda \geq 0$ such that 
\begin{align*}
 t^{2} =\lambda \left(Id-\eta \otimes \xi \right), \quad g(tX, Y)=-g(X, tY)
\end{align*}
for any nonzero vectors $X,Y \in \mathfrak{D_{\lambda}}$ at $p \in M$ and not proportional to $\xi_{p}$. Here, $\lambda$ is a \textit{slant coefficient} of $M$. 
\end{definition}
\begin{rem}
It is important to note that the invariant and anti-invariant immersions with slant coefficent $\lambda=1$ and $\lambda=0$ respectively.  A slant immersion which is neither invariant nor anti-invariant is called a \textit{proper slant} immersion. 
\end{rem} 
\noindent If we denote the orthogonal distribution to $\xi \in \Gamma(TM)$ by $\mathbb{D}$, then the tangent bundle of $M$ is given as follows: 
\begin{align*}
TM=\mathbb{D}\,\, \oplus <\xi>.
\end{align*}

\section{$\mathcal{P}\mathcal{R}$-anti-slant submanifolds}\label{pranti-sub}
\noindent In this section, we define $\mathcal{P}\mathcal{R}$-{\it anti-slant} submanifolds of an almost paracontact pseudo-Riemannain metric manifold and derive  characterization results for the same. 
\begin{definition}
Let $M$ is an isometrically immersed submanifold of an almost paracontact manifold $\widetilde{M}(\varphi,\xi,\eta,g)$ such that the characteristic vector field $\xi \in \Gamma(TM)$. Then $M$ is said to be a  $\mathcal{P}\mathcal{R}$-{\it anti-slant submanifold} if it    is furnished with the pair of non-degenerate orthogonal distribution ($\mathfrak{D}^{\bot}, \mathfrak{D}_{\lambda}$) satisfies the following conditions:
\begin{itemize}
\item[(i)] $TM = \mathfrak{D}^{\bot} \oplus \mathfrak{D}_{\lambda} \oplus <\xi>$, 
\item[(ii)] the distribution $\mathfrak{D}{^\bot}$ is anti-invariant under $\varphi$, i.e., $\varphi(\mathfrak{D}{^\bot})\subset \Gamma(TM){^\bot}$ and
\item[(iii)] the distribution $\mathfrak{D}_{\lambda}$ is slant distribution with slant coefficient $\lambda$.  
\end{itemize}
We say that a $\mathcal{P}\mathcal{R}$- anti-slant submanifold is \textit{proper} if $ \mathfrak{D}^{\bot} \neq \{0\}$,  $\mathfrak{D}_{\lambda}\neq \{0\}$ and $\lambda \neq 0, 1$.
A $\mathcal{P}\mathcal{R}$- anti-slant submanifold is said to be \textit{mixed totally geodesic} if $h(X, Z)=0$ for all $X \in \Gamma(\mathfrak{D}_{\lambda})$ and $Z \in \Gamma(\mathfrak{D}^{\bot}\oplus <\xi>)$.
\end{definition}
\noindent Now, we can give the following important corollary as a straight forward consequences of the definition of slant submanifold  of an almost paracontact manifolds:
\begin{proposition}\label{coro_slant}
Let $M$ be a slant submanifold of an almost paracontact metric manifold $\widetilde{M}(\varphi,\xi,\eta,g)$ with $\xi \in  \Gamma(TM)$. Then
\begin{align} 
g(tX,tY)&=\lambda g(\varphi X,\varphi Y), \label{gtcos}\\ 
g(nX,nY)&=(1-\lambda) g(\varphi X,\varphi Y), \label{gnsin}
\end{align}
for any $X,Y\in \Gamma(\mathfrak{D}_{\lambda})$. 
\end{proposition}
\begin{proof}
Since $g(tX,tY)=-g(X,t^{2}Y)=-\lambda g(X,\varphi^{2}Y)=\lambda g(\varphi X,\varphi Y)$. Therefore, by the use of Eq. \eqref{metric} and definition \ref{slantdfn}, we get Eq. \eqref{gtcos}. Eq. \eqref{gnsin} follows from Eqs. \eqref{phix} and \eqref{gtcos}. This completes the proof.
\end{proof}

\begin{proposition}\label{pranti-thm}
Let $M$ be a immersed submanifold of an almost paracontact metric manifold $\widetilde{M}(\varphi,\xi,\eta,g)$ such that $\xi \in \Gamma(TM)$. Then for any $X, Y \in \Gamma(\mathfrak{D}_{\lambda})$, we have
\begin{itemize}
\item[(i)] $t'nX=(1-\lambda)(X-\eta(X)\xi)$ and 
\item[(ii)] $n'nX=-ntX$.
\end{itemize}
\end{proposition}
\begin{proof}
We have from Eq. \eqref{phin} that $\varphi nX=t'nX+n'nX$. Then, taking inner product with $Y$ and applying Eqs. \eqref{phieta} and \eqref{gnsin} we derive the formula-(i). For formula-(ii), we have from Eq. \eqref{phix} and definition \ref{slantdfn} that
\begin{align}\label{antimain-1}
 \varphi tX=t^{2}X+ntX=\lambda(X-\eta(X)\xi) + ntX.
\end{align}
On the other hand, using Eq. \eqref{phin} and formula-(i) it can be written that
\begin{align}\label{antimain-2}
 \varphi nX=t'nX+n'nX=(1-\lambda) (X-\eta(X)\xi) + n'nX.
\end{align}
From Eqs. \eqref{antimain-1} and \eqref{antimain-2} we get
$\varphi tX+\varphi nX=(X-\eta(X)\xi)+ntX+n'nX$. Formula-(ii) can be achieved by employing Eqs. \eqref{phieta} and \eqref{phix} in previous expression. This completes the proof.
\end{proof}

\noindent Now, we give the necessary and sufficient condition for integrability and totally geodesic foliation of distributions equipped with the submanifold $M$.
\begin{theorem}\label{pranti-thm1}
Let $M$ be a proper $\mathcal{P}\mathcal{R}$-anti-slant submanifold of a nearly paracosymplectic manifold $\widetilde{M}(\varphi,\xi,\eta,g)$. Then the distribution $\mathfrak{D}_{\lambda}$, is integrable if and only if 
\begin{align}\label{1}
2\lambda g(\widetilde\nabla_{X}Y,Z)=g(A_{ntY}X,Z)+g(A_{ntX}Y,Z)-g(A_{\varphi Z}tY,X)-g(A_{\varphi Z}tX,Y). 
\end{align}
for any $ Z \in \Gamma(\mathfrak{D}^{\bot} \oplus \xi)$ and $X, Y \in \Gamma(\mathfrak{D}_{\lambda})$.
\end{theorem}
\begin{proof}
From the fact that $\xi$ is killing and Eq. \eqref{metric}, we have
\begin{align}\label{pranti-thm1-1}
g([X,Y],Z)=g(\widetilde\nabla_{X}Y,Z)-g(\widetilde\nabla_{Y}X,Z)=-g(\varphi\widetilde\nabla_{X}Y,\varphi Z)-g(\widetilde\nabla_{Y}X, Z).
\end{align} 
Using Eqs. \eqref{npcos} and \eqref{covphi} in Eq. \eqref{pranti-thm1-1}, we obtain that
\begin{align}\label{pranti-thm1-2}
g([X,Y], Z) &=-g(\widetilde\nabla_{X} \varphi Y-(\widetilde\nabla_{X}\varphi) Y,\varphi Z)-g(\widetilde\nabla_{Y}X, Z)\nonumber \\
            &=-g(\widetilde\nabla_{X}{tY}, \varphi Z)-g(\widetilde\nabla_{X}{nY}, \varphi Z)-g((\widetilde\nabla_{Y}\varphi) X,\varphi Z)-
               g(\widetilde\nabla_{Y}X, Z).
\end{align}
Applying Eqs. \eqref{cplevi} and \eqref{gauss} in equation \eqref{pranti-thm1-2}, we get
\begin{align*}
g([X,Y], Z) &=-g(h(X,tY), \varphi Z)+g(\widetilde\nabla_{X}\varphi Z, nY) \\
               &-g(\widetilde\nabla_{Y}\varphi X-\varphi \widetilde\nabla_{Y}X, \varphi Z)-g(\widetilde\nabla_{Y}X, Z).
\end{align*}
By the use of Eqs. \eqref{phieta}, \eqref{metric}, \eqref{phix} and covariant differentiation of $\varphi$, the above expression reduced to
\begin{align}\label{pranti-thm1-4}
g([X,Y], Z) &=-g(h(X,tY), \varphi Z)+g((\widetilde\nabla_{X}\varphi) Z+\varphi\widetilde\nabla_{X}Z, nY) \nonumber \\
               &-g(\widetilde\nabla_{Y}(tX+nX), \varphi Z)-2g(\widetilde\nabla_{Y}X, Z).
\end{align}
Employing Eqs. \eqref{phin}, \eqref{phitn} and the fact that structure is nearly paracosymplectic in \eqref{pranti-thm1-4}, we recieve that
\begin{align}\label{pranti-thm1-5}
g([X,Y], Z) &=-g(h(X,tY), \varphi Z)-g(\widetilde\nabla_{X}Z, t'nY+n'nY) \nonumber \\
               & -g(h(Y,tX), \varphi Z)-g(\widetilde\nabla_{Y}Z, t'nX+n'nX)-2g(\widetilde\nabla_{Y}X, Z).
\end{align}
In light of proposiotion \ref{pranti-thm} and Eq. \eqref{cplevi}, equation \eqref{pranti-thm1-5} yields
\begin{align*}
g([X,Y], Z) &=-g(h(X,tY), \varphi Z)+(1-\lambda) g(\widetilde\nabla_{X}Y,Z)+g(\widetilde\nabla_{X}Z, ntY) \nonumber \\
               & -g(h(Y,tX), \varphi Z)+(1-\lambda)g(\widetilde\nabla_{Y}X,Z)+g(\widetilde\nabla_{Y}Z, ntX)-2g(\widetilde\nabla_{Y}X, Z).
\end{align*}
From above equation, we conclude that 
\begin{align}\label{pranti-thm1-7}
2\lambda g(\widetilde\nabla_{X}Y,Z)&=-g(h(X,tY), \varphi Z)+g(\widetilde\nabla_{X}Z, ntY) \nonumber \\
               & -g(h(Y,tX),\varphi Z)+g(\widetilde\nabla_{Y}Z, ntX).
\end{align}
By the virtue of Eqs. \eqref{shp2form} and \eqref{pranti-thm1-7}, we obtain equation \eqref{1}. This completes the proof of the theorem. 
\end{proof}

\begin{theorem}\label{pranti-thm2}
Let $M$ be a proper $\mathcal{P}\mathcal{R}$-anti-slant submanifold of a nearly paracosymplectic manifold $\widetilde{M}(\varphi,\xi,\eta,g)$. Then the distribution $(\mathfrak{D}^{\bot} \oplus \xi)$ defines a totally geodesic foliation if and only if 
\begin{align}\label{2}
2g(A_{ntX}Z, W)= g(A_{\varphi W}Z,tX)+g(A_{\varphi Z}W, tX) 
\end{align}
for any $Z,W \in \Gamma(\mathfrak{D}^{\bot} \oplus \xi)$ and $X \in \Gamma(\mathfrak{D}_{\lambda})$.
\end{theorem}
\begin{proof}
We have from Eq. \eqref{metric} and the fact that $X$ and $\xi$ are orthogonal that
\begin{align}\label{pranti-thm2-1}
g(\widetilde\nabla_{Z}W, X)=-g(\varphi\widetilde\nabla_{Z}W, \varphi X).
\end{align}
Using Eqs. \eqref{antisymphi}, \eqref{phix} and \eqref{covphi} in Eq. \eqref{pranti-thm2-1}, we get
\begin{align}\label{pranti-thm2-2}
g(\widetilde\nabla_{Z}W, X)&=-g(\varphi\widetilde\nabla_{Z}W, tX)-g(\varphi\widetilde\nabla_{Z}W, nX) \nonumber \\
                  &=-g(\widetilde\nabla_{Z}\varphi W-(\widetilde\nabla_{Z}\varphi)W, tX)+g(\widetilde\nabla_{Z}W, \varphi nX).
\end{align}
Employing Eqs. \eqref{phitn}, \eqref{cplevi} and proposition \ref{pranti-thm} in \eqref{pranti-thm2-2}, we obtain that
\begin{align*}
g(\widetilde\nabla_{Z}W, X)	&=-g(\widetilde\nabla_{Z}\varphi W,tX)+g((\widetilde\nabla_{Z}\varphi)W, tX)+ g(\widetilde\nabla_{Z}W,t'nX)   
                                                           +g(\widetilde\nabla_{Z}W, n'nX) \nonumber \\ 
                                                    &= g(\widetilde\nabla_{Z}tX, \varphi W)+g((\widetilde\nabla_{Z}\varphi)W, tX)+(1-\lambda)g(\widetilde
                                                         \nabla_{Z}W,X) -g(\widetilde\nabla_{Z}W, ntX).
\end{align*}
Above expression on using nearly paracosymplectic structure and Eq. \eqref{covphi} reduced to 
\begin{align}\label{pranti-thm2-4}
g(\widetilde\nabla_{Z}W, X)&= g(\widetilde\nabla_{Z}tX, \varphi W)-g((\widetilde\nabla_{W}\varphi)Z, tX)+(1-\lambda) 
                                                      g(\widetilde\nabla_{Z}W,X)\nonumber \\ &-g(\widetilde\nabla_{Z}W, ntX)\nonumber \\
                                                &= g(\widetilde\nabla_{Z}tX, \varphi W)-g(\widetilde\nabla_{W}\varphi Z-\varphi\widetilde\nabla_{W}Z, tX)+(1- 
                                                        \lambda) g(\widetilde\nabla_{Z}W,X)\nonumber \\ &-g(\widetilde\nabla_{Z}W, ntX).          
\end{align}
From above expression, we conclude that
\begin{align}\label{pranti-thm2-5}
\lambda g(\widetilde\nabla_{Z}W, X) &= g(\widetilde\nabla_{Z}tX, \varphi W)-g(\widetilde\nabla_{W}\varphi Z,tX)-g(\widetilde\nabla_{W}Z, t^{2}X+ntX) 
                                               \nonumber \\&-g(\widetilde\nabla_{Z}W, ntX).          
\end{align}
By the use of Eqs. \eqref{cplevi}, \eqref{shp2form} and the fact that $\eta(X)=0$, we obtain from Eq. \eqref{pranti-thm2-5} that
\begin{align*}
\lambda g(\widetilde\nabla_{Z}W, X) &= g(h(Z,tX), \varphi W)+g(h(W, tX), \varphi Z)-\lambda g(\widetilde\nabla_{W}Z, X) \\ 
                 &-2g(h(W,Z), ntX).          
\end{align*}
By the virtue of Eq. \eqref{shp2form} the above expression yields Eq. \eqref{2}.
\end{proof}

\section{$\mathcal{P}\mathcal{R}$-anti-slant warped product of the form $F\times_{f}N_{\lambda}$}\label{prwanti}
\noindent Let $\left(B,g_{B} \right)$ and $\left(F ,g_{F} \right)$ be two pseudo-Riemannian manifolds and ${f}$ be a positive smooth function on $B$. Consider the product manifold $B\times F$ with canonical projections 
\begin{align}\label{cp}
\pi:B \times F\to B\quad{\rm and}\quad \sigma:B \times F\to F.
\end{align}
Then the manifold $M=B \times_{f} F $ is said to be \textit{warped product} if it is equipped with the following warped metric
\begin{align}\label{wmetric}
g(X,Y)=g_{B}\left(\pi_{\ast}(X),\pi_{\ast}(Y)\right) +(f\circ\pi)^{2}g_{F}\left(\sigma_{\ast}(X),\sigma_{\ast}(Y)\right)
\end{align}
for all $X,Y\in TM$ and $\ast$ stands for derivation map, or equivalently,
\begin{align}
g=g_{B} +f^{2} g_{F}.
\end{align}
The function $f$ is called {\it the warping function} and a warped product manifold $M$ is said to be trivial if $f$ is constant (see also \cite{RB,SA,ts}).
Now, we recall the following proposition for the warped product manifolds \cite{RB}:
\begin{proposition}$\label{propwp}$
For $X, Y \in \Gamma(TB)$ and $Z \in \Gamma(TF)$, we obtain on warped product manifold $M=B\times_{f}F$ that
\begin{itemize}
\item[(i)] 	$\nabla _{X}Y \in \Gamma(TB),$
\item[(ii)]	$\nabla _{X}Z =\nabla _{Z}X=X(\ln{f})Z,$
\end{itemize}
where $\nabla$ denotes the Levi-civita connections on $M$.
\end{proposition}
\noindent For a warped product $M=B \times_{f}F$, $B$ is totally geodesic and $F$ is totally umbilical in $M$ \cite{RB}.
\begin{definition}
A $\mathcal{P}\mathcal{R}$- anti-slant submanifold is called a $\mathcal{P}\mathcal{R}$-{\it anti-slant warped product} if it is a warped product of the form: $F\times_{f}N_{\lambda}$, where $F$ is an anti invariant submanifold, $N_{\lambda}$ is a proper slant submanifold of an almost paracontact manifold $\widetilde{M}(\varphi,\xi,\eta,g)$ with slant coefficient $\lambda$ and  $f$ is a non-constant positive function on $F$. If $f$ is constant then the product  of the form:  $F\times_{f}N_{\lambda}$ is called $\mathcal{P}\mathcal{R}$- {\it anti-slant product}.
\end{definition}
In this section, we shall examine $\mathcal{P}\mathcal{R}$-anti-slant warped product submanifolds of a nearly paracosymplectic manifold such that $\xi \in \Gamma(TN_{\lambda})$ and $\xi \in \Gamma(TF)$.

\noindent\textit{Case-$1$: When characteristic vector field $\xi$ is tangent to $N_{\lambda}$.}
\begin{proposition}\label{prop5.1}
There do not exist a $\mathcal{P}\mathcal{R}$-anti-slant warped product submanifold $M=F \times_{f}N_{\lambda}$ of a nearly paracosymplectic manifold $\widetilde{M}(\varphi,\xi,\eta,g)$.
\end{proposition}
\begin{proof}
From proposition \ref{killing} and Eq. \eqref{gauss} we have $\nabla_{Z}\xi+h(Z ,\xi)=0$. Comparing  tangential part and  using proposition \ref{propwp} we obtain that $Z \ln f=0 $. Thus $f$ is constant, since $Z \in \Gamma(TF)$ is non null vector field. This completes the proof.
\end{proof}

\noindent \textit{Case-$2$: When characteristic vector field $\xi$ is tangent to $F$.}

Here, we first give an example illustrating $\mathcal{P}\mathcal{R}$-anti-slant warped product submanifold $M=F\times_{f}N_{\lambda}$ of a nearly paracosymplectic manifold $\widetilde{M}(\phi,\xi,\eta,g)$ and then prove an important lemma for later use. 
\begin{example}\label{exa1}
 Let $\widetilde M=\mathbb{R}^4\times\mathbb{R}_{+}\subset\mathbb{R}^5$ be a $5$-dimensional manifold with the standard Cartesian coordinates $(x_{1}, x_{2}, y_{1}, y_{2}, t)$. Define the nearly paracosymplectic pseudo-Riemannian metric structure $(\varphi,\xi,\eta, g)$ on $\widetilde M$ by
 \begin{align}\label{strcex2}
 \varphi\left(\frac{\partial}{\partial x_{i}}\right)&=\frac{\partial}{\partial y_{i}},\,
 \varphi\left(\frac{\partial}{\partial y_{i}}\right)=\frac{\partial}{\partial x_{i}},\,
 \varphi \left(\frac{\partial}{\partial t}\right)=0,\,\xi=\frac{\partial}{\partial t},\,\eta=dt, \\
  g&=\sum(dx_{i})^{2}-\sum(dy_{i})^{2}+(dt)^{2},\ \forall i\in\{1,2\}.\nonumber
 \end{align}
Now, let $M$ is an isometrically immersed smooth submanifold in $\mathcal{R}^{5}$ defined by 
 \begin{align}\label{strucsub}
 x_{1}=v\cosh\alpha,\,x_{2}=v\cosh\beta,\, y_{1}=v\sinh\alpha,\,y_{2}=v\sinh\beta, t=u,
 \end{align}
 where $v \in \RR-\{0,1\}$.
 Then the $TM$ spanned by the vectors
 \begin{align} \label{tanbundl}
Z_{1}&=\cosh\alpha\frac{\partial}{\partial x_{1}}+\cosh\beta\frac{\partial}{\partial                                                               
      x_{2}}+\sinh\alpha\frac{\partial}{\partial y_{1}}+\sinh\beta\frac{\partial}{\partial y_{2}}, \nonumber\\
Z_{2}&=v\sinh\alpha\frac{\partial}{\partial x_{1}}+v\cosh\alpha\frac{\partial}{\partial y_{1}},\\
Z_{3}&=v\sinh\beta\frac{\partial}{\partial x_{2}} + v\cosh\beta\frac{\partial}{\partial y_{2}},\,\, Z_{4}=\frac{\partial}{\partial t}, \nonumber
 \end{align}
 where $Z_{1}, Z_{2}, Z_{3}, Z_{4}  \in \Gamma(TM)$.
 Therefore from Eq. \eqref{strcex2}, we find that
\begin{align}\label{phibundl} 
\varphi(Z_{1})&=\sinh\alpha\frac{\partial}{\partial x_{1}}+\sinh\beta\frac{\partial}{\partial        
        x_{2}}+\cosh\alpha\frac{\partial}{\partial y_{1}}+\cosh\beta\frac{\partial}{\partial                                                                  
        y_{2}}, \nonumber\\
\varphi(Z_{2})&=v\cosh\alpha\frac{\partial}{\partial x_{1}}+v\sinh\alpha\frac{\partial}{\partial y_{1}},\\
\varphi(Z_{3})&=v\cosh\beta\frac{\partial}{\partial x_{2}}+v\sinh\beta\frac{\partial}{\partial y_{2}},\,\,
\varphi(Z_{4})=0. \nonumber
\end{align}
From Eqs. \eqref{tanbundl} and \eqref{phibundl}, we obtain that $\mathfrak{D_{\lambda}}$ is a proper slant distribution given by span\{$ Z_{1}, Z_{2}, Z_{3}$\} with slant coefficent $\lambda=\frac{1}{2}$ and  $\mathfrak{D}^{\bot}$ is an anti-invariant distribution given by span\{$Z_{4}$\} with dimension not equal to zero, where $\xi = Z_{4}$ and $\varphi(Z_{4})=0$, and $\eta({Z_{4}})=1$. Therefore, $M$ is a proper $\mathcal{P}\mathcal{R}$-anti-slant submanifold of a nearly paracosymplectic manifold $\widetilde M$.
Here, the induced pseudo-Riemannian metric tensor $g$ of $M$ is given by $$g= dt^{2}+v^{2}\{\frac{2}{v^2}dv^{2}-d\alpha^{2}-d\beta^{2}\}=g_{F}+v^{2}g_{N_{\lambda}}.$$ Hence $M$ is a $4$-dimensional $\mathcal{P}\mathcal{R}$-anti-slant warped product of $\mathcal{R}^{5}$ with $f=v^{2}$. 
\end{example}
\begin{example}\label{exa2}
 Let $\widetilde M=\mathbb{R}^6\times\mathbb{R}_{+}\subset\mathbb{R}^7$ be a $7$-dimensional manifold with the standard Cartesian coordinates $(x_{1}, x_{2}, x_{3}, y_{1}, y_{2}, y_{3}, t)$. Define the nearly paracosymplectic pseudo-Riemannian metric structure $(\varphi,\xi,\eta, g)$ on $\widetilde M$ by
 \begin{align}\label{2strcex2}
 \varphi\left(\frac{\partial}{\partial x_{i}}\right)&=\frac{\partial}{\partial y_{i}},\,
 \varphi\left(\frac{\partial}{\partial y_{i}}\right)=\frac{\partial}{\partial x_{i}},\,
 \varphi \left(\frac{\partial}{\partial t}\right)=0,\,\xi=\frac{\partial}{\partial t},\,\eta=dt, \\
  g&=\sum(dx_{i})^{2}-\sum(dy_{i})^{2}+(dt)^{2},\ \forall i\in\{1,2,3\}.\nonumber
 \end{align}
Now, let $M$ is an isometrically immersed  smooth submanifold in $\mathcal{R}^{7}$ defined by 
 \begin{align}\label{2strucsub}
\chi(u,\alpha, v, t)=\left(\frac{u}{\sqrt{2}}\cosh(\alpha), u + v, v, \frac{u}{\sqrt{2}}\sinh(\alpha), k_{1}, k_{2}, t \right)
 \end{align}
 where $k_{1}, k_{2}$ are constants and $u \in \RR-\{0\}$. Then the $TM$ spanned by the vectors
 \begin{align} \label{2tanbundl}
Z_{1}&=\frac{1}{\sqrt{2}}\cosh(\alpha)\frac{\partial}{\partial x_{1}}+\frac{\partial}{\partial x_{2}}+\frac{1}
        {\sqrt{2}}\sinh(\alpha)\frac{\partial}{\partial y_{1}}, \nonumber\\
Z_{2}&=\frac{u}{\sqrt{2}}\sinh(\alpha)\frac{\partial}{\partial x_{1}}+\frac{u}{\sqrt{2}}\cosh(\alpha)\frac{\partial}{\partial y_{1}}, \\
Z_{3}&=\frac{\partial}{\partial x_{2}} + \frac{\partial}{\partial x_{3}},\,\, Z_{4}=\frac{\partial}{\partial t}, \nonumber
 \end{align}
 where $Z_{1}, Z_{2}, Z_{3}, Z_{4}  \in \Gamma(TM)$.
 Using Eq. \eqref{2strcex2}, we obtain that
\begin{align}\label{2phibundl} 
\varphi(Z_{1})&=\frac{1}{\sqrt{2}}\sinh(\alpha)\frac{\partial}{\partial x_{1}}+\frac{\partial}{\partial y_{2}}+\frac{1} 
                {\sqrt{2}}\cosh(\alpha)\frac{\partial}{\partial y_{1}}, \nonumber\\
\varphi(Z_{2})&=\frac{u}{\sqrt{2}}\cosh(\alpha)\frac{\partial}{\partial x_{1}}+\frac{u}{\sqrt{2}}\sinh(\alpha)\frac{\partial}{\partial 
                 y_{1}}, \\
\varphi(Z_{3})&=\frac{\partial}{\partial y_{2}} + \frac{\partial}{\partial y_{3}},\,\, \varphi(Z_{4})=0, \nonumber
\end{align}
From Eqs. \eqref{2tanbundl} and \eqref{2phibundl} we can find that $\mathfrak{D_{\lambda}}$ is a proper slant distribution defined by span\{$ Z_{1}, Z_{2}$\} with slant coefficient $\lambda=\frac{1}{3}$ and  $\mathfrak{D}^{\bot}$ is an anti-invariant distribution defined by span\{$Z_{3}, Z_{4}$\} with dimension not equal to zero, where $\xi = Z_{4}$ and $\eta({Z_{4}})=1$. So, $M$ turn into a proper $\mathcal{P}\mathcal{R}$-anti-slant submanifold. Here, the induced pseudo-Riemannian non-degenerate metric tensor $g$ of $M$ is specified by
$$g= dt^{2}+2dv^{2}+\frac{1}{2}u^{2}\{3/u^{2}du^{2}-d\alpha^{2}\}=g_{F}+\frac{1}{2}u^{2}g_{N_{\lambda}}.$$
Thus, $M$ is a $4$-dimensional $\mathcal{P}\mathcal{R}$-anti-slant warped product submanifold of $\mathcal{R}^{7}$ with wrapping function $f=\frac{1}{2}u^{2}$.
\end{example}
\begin{lemma}\label{lem}
If $M=F\times_{f}N_{\lambda}$ be a $\mathcal{P}\mathcal{R}$-anti-slant warped product submanifold of a nearly paracosymplectic manifold $\widetilde{M}(\varphi,\xi,\eta,g)$ then for all $X$ tangent to $N_{\lambda}$ and $Z$ tangent to $F$, we have
\begin{itemize}
\item[$(a)$] $2g(A_{ntX}X, Z)=g(A_{\varphi Z}X,tX)+ g(A_{nX}Z, tX)-(Z\ln f)\lambda\vert\vert X \vert\vert^{2}$\ and
\item[$(b)$] $g(A_{nX}Z, tX)=2g(A_{\varphi Z}X,tX)-g(A_{ntX}X, Z)$.
\end{itemize}

\end{lemma}
\begin{proof}
From Gauss formula and Eq. \eqref{shp2form}, we can write that
\begin{align}\label{lem1}
g(A_{ntX}X,Z)=g(\widetilde\nabla_{Z}X-\nabla_{Z}X, ntX).
\end{align}
Using Eqs. \eqref{phix}, \eqref{covphi}, definition \ref{slantdfn} and the fact that $\xi$ is orthogonal to $X$ we get
\begin{align*}
g(A_{ntX}X,Z)&=g(\widetilde\nabla_{Z}X,\varphi tX-t^{2}X)=-g(\varphi\widetilde\nabla_{Z}X,tX)-g(\widetilde\nabla_{Z}X, t^{2}X) \nonumber \\
             &=-g(\widetilde\nabla_{Z}\varphi X-(\widetilde\nabla_{Z}\varphi)X, tX)-\lambda g(\widetilde\nabla_{Z}X, X).
\end{align*}
Above equation by applying definition of nearly paracosymplectic and proposition \ref{propwp} reduced to
\begin{align}\label{lem3}
g(A_{ntX}X,Z)&=-g(\widetilde\nabla_{Z}\varphi X+(\widetilde\nabla_{X}\varphi)Z, tX)-(Z\ln f) \lambda \vert\vert X \vert\vert^{2}.
\end{align}
Therefore, again using Eqs. \eqref{phix}, \eqref{covphi} and proposition \ref{coro_slant} in \eqref{lem3}, we obtain that
\begin{align}\label{lem4}
g(A_{ntX}X,Z)&=-g(\widetilde\nabla_{Z}nX, tX)-g(\widetilde\nabla_{X}\varphi Z-\varphi\widetilde\nabla_{X} Z, tX) \nonumber \\
             &= g(\widetilde\nabla_{Z}tX, nX)-g(\widetilde\nabla_{X}\varphi Z, tX)-g(\widetilde\nabla_{X} Z, \varphi tX).
\end{align}
Employing Eqs. \eqref{phix}, \eqref{shp2form}, definition \ref{slantdfn} and proposition \ref{propwp} in equation \eqref{lem4}, we find the formula-$(a)$. 
For formula-$(b)$: we have from Eq. \eqref{shp2form} that
 \begin{align}\label{lem5}
g(A_{nX}tX,Z)=g(h(tX, Z), nX).
\end{align} 
Applying Gauss formula, Eqs. \eqref{phix} and definition \ref{slantdfn}, we find from above equation that
\begin{align}\label{lem6}
g(A_{nX}tX,Z)=-g(\varphi \widetilde\nabla_{Z}tX, X)+\lambda g(\nabla_{Z}X, X).
\end{align} 
Using proposition \ref{propwp}, Eqs. \eqref{covphi} and \eqref{npcos}   in equation \eqref{lem6}, we obtain that
\begin{align}\label{lem7}
g(A_{nX}tX,Z)=-g(\widetilde\nabla_{Z}\varphi tX, X)-g((\widetilde\nabla_{tX}\varphi) Z, X)+(Z\ln f)\lambda\vert\vert X \vert\vert^{2}.
\end{align}
Employing Eq. \eqref{covphi}, definition \ref{slantdfn} and proposition \ref{propwp} in equation \eqref{lem7}, we get
\begin{align}\label{lem8}
g(A_{nX}tX,Z)=-g(\widetilde\nabla_{Z}ntX, X)-g(\widetilde\nabla_{tX}\varphi Z, X)+g(\widetilde\nabla_{tX}Z, \varphi X).
\end{align}
Again using Eqs. \eqref{cplevi}, \eqref{phix}, \eqref{shp2form} and Gauss-Weingarten formula we achieve from equation \eqref{lem8} that
\begin{align}\label{lem9}
2g(h(tX,Z), nX)=g(h(Z, X), ntX)+g(h(tX, X), \varphi Z)-g(\nabla_{Z}tX, tX).
\end{align} 
Equation \eqref{lem9} by the virtue of proposition \ref{propwp} and Eq. \eqref{gtcos} reduced to
\begin{align}\label{lem10}
2g(h(tX,Z), nX)=g(h(Z, X), ntX)+g(h(tX, X), \varphi Z)+(Z\ln f)\lambda g(X,X).
\end{align}
Interchanging $X$ by $tX$ and using definition \ref{slantdfn}, proposition \ref{propwp} in equation \eqref{lem10}, we have
\begin{align}\label{lem11}
2g(h(X,Z), ntX)=g(h(Z, tX), nX)+g(h(tX, X), \varphi Z)-(Z\ln f)\lambda\vert\vert X \vert\vert^{2}.
\end{align}
Thus, formula-$(b)$ follows from Eqs. \eqref{lem10}, \eqref{lem11} and \eqref{shp2form}.
\end{proof}

\begin{lemma}\label{5.1}
If $M=F\times_{f}N_{\lambda}$ be a $\mathcal{P}\mathcal{R}$-anti-slant warped product submanifold of a nearly paracosymplectic manifold $\widetilde{M}(\varphi,\xi,\eta,g)$ then
\begin{align}\label{L5.1_1}
g(\mathcal{T}_{X}tX, Z)=g(A_{ntX}X, Z)-g(A_{nX}Z, tX)
\end{align}
for all $X$ is tangent to $N_{\lambda}$ and $Z$ is tangent to $F$.
\end{lemma}
\begin{proof}
By virtue of Eqs. \eqref{gauss}, \eqref{phix} and \eqref{covphi} we obtain that
\begin{align*}
g(A_{nX}Z, tX)=g(tX,\nabla_{Z}tX)-g(tX, (\widetilde\nabla_{Z\varphi})X)-g(tX,\varphi \widetilde\nabla_{Z}X).
\end{align*}
By use of Eqs. \eqref{antisymphi}, \eqref{phitn} and proposition \ref{propwp}, the above equation reduced to
\begin{align}\label{L5.1extra}
g(A_{nX}Z, tX)=(Z\ln f)g(tX,tX)-g(tX,\mathcal{T}_{Z}X)+g(\varphi(tX), \widetilde\nabla_{Z}X).
\end{align}
Employing Eqs. \eqref{metric}, \eqref{tantan}, \eqref{gtcos} and Gauss formula in Eq. \eqref{L5.1extra} we have
\begin{align}\label{L5.1extra1}
g(A_{nX}Z, tX)=&-(Z\ln f)\lambda g(X,X)+g(tX,\mathcal{T}_{X}Z)+g(t^{2}X, \nabla_{Z}X)\nonumber\\ &+g(A_{ntX}X, Z).
\end{align}
In light of proposition \ref{propwp}, Eq. \eqref{P7} and definition of slant submaifold, Eq. \eqref{L5.1extra1} yields Eq. \eqref{L5.1_1}. This completes the proof.
\end{proof}
\noindent Now, we prove:
\begin{theorem}\label{thrm5.1}
Let  $M=F\times_{f}N_{\lambda}$ is a $\mathcal{P}\mathcal{R}$-anti-slant warped product submanifold of a nearly paracosymplectic manifold $\widetilde{M}(\varphi,\xi,\eta,g)$ with $\xi \in \Gamma(TF)$. Then $M$ is a $\mathcal{P}\mathcal{R}$-anti-slant product if and only if $\mathcal{T}_{X}tX$ is tangent to $N_{\lambda}$, for all $X \in \Gamma(TN_{\lambda})$ and $Z \in \Gamma(TF)$.
\end{theorem}
\begin{proof}
From Eqs. \eqref{antisymphi}, \eqref{gauss}, \eqref{shp2form}, \eqref{covphi} and \eqref{phitn}, we have
\begin{align*}
g(A_{nZ}X, tX)&=g(\widetilde\nabla_{tX}X,\varphi Z)=g(\mathcal{T}_{tX}X, Z)+ g(\varphi X, \widetilde\nabla_{tX}Z). 
\end{align*} 
From Gauss fromula and Eqs. \eqref{phix}, \eqref{tantan}, we obtain that
$$g(A_{nZ}X, tX)=-g(\mathcal{T}_{X}tX, Z)+g(tX, \nabla_{tX}Z)+g(A_{nX}tX, Z).$$
Then, using Eq. \eqref{gtcos} and proposition \ref{propwp}, we get
\begin{align}\label{T5.1_1}
g(A_{nZ}X, tX)=-(Z\ln f)\lambda g(X,X)-g(\mathcal{T}_{X}tX, Z)+g(A_{nX}tX, Z).
\end{align}
Replacing $X$ by $tX$ and, using Eqs. \eqref{P7}, \eqref{tantan}, definition of slant submaifold,   
\eqref{gtcos} in Eq. \eqref{T5.1_1}, we arrive at
\begin{align}\label{T5.1_2}
g(A_{nZ}X, tX)=(Z\ln f)\lambda\vert\vert X \vert\vert^{2}+g(\mathcal{T}_{X}tX, Z)+g(A_{ntX}X, Z).
\end{align}
Thus, from Eqs. \eqref{T5.1_1}, \eqref{T5.1_2} and lemma \eqref{L5.1_1}, we derive
\begin{align}\label{T5.1_3}
g(\mathcal{T}_{X}tX, Z)=\frac{2}{3}(Z\ln f)\lambda\vert\vert X \vert\vert^{2}.
\end{align}
Eq. \eqref{T5.1_3} implies that, $(Z\ln f)=0$ if and only if $\mathcal{T}_{X}tX$ is tangent to $N_{\lambda}.$ Since, $X,\ Z$ are non-null vector fields and the fact that $N_{\lambda}$ is proper slant submanifold. This completes the proof.
\end{proof}

\begin{theorem}\label{thm1}
Let $M=F\times_{f}N_{\lambda}$ be a mixed totally geodesic $\mathcal{P}\mathcal{R}$-anti-slant warped product submanifold of a nearly paracosymplectic manifold $\widetilde{M}(\varphi,\xi,\eta,g)$ with $\xi \in \Gamma(TF)$. Then $M$ is a $\mathcal{P}\mathcal{R}$-anti-slant warped product submanifold, for any  $X \in \Gamma(TN_{\lambda})$ and $Z \in \Gamma(TF)$.
\end{theorem}
\begin{proof}
From formula-$(a)$ and formula-$(b)$ of lemma \ref{lem}, we obtain that
\begin{align}\label{thm1-1}
3g(A_{ntX}X, Z)= 3g(A_{\varphi Z}X, tX)-(Z\ln f)\lambda\vert\vert X \vert\vert^{2}.
\end{align}
Employing Eq. \eqref{shp2form} in \eqref{thm1-1} we find that
\begin{align*}
3g(h(X, Z), ntX)= 3g(A_{\varphi Z}X, tX)-(Z\ln f)\lambda\vert\vert X \vert\vert^{2}.
\end{align*}
Since $M$ is mixed totally geodesic, so we achieve from above equation that
\begin{align}\label{thm1-3}
g(A_{\varphi Z}X, tX)=\frac{1}{3}\{ (Z\ln f)\lambda\vert\vert X \vert\vert^{2} \}.
\end{align}
This completes the proof of the theorem. 
\end{proof}

\begin{theorem}\label{thm2}
 Let $M\to \widetilde{M}$ be an isometric immersion of a $\mathcal{P}\mathcal{R}$-anti-slant submanifold $M$ into a nearly paracosymplectic manifold $\widetilde{M}(\varphi,\xi,\eta,g)$. Then $M$ is locally $\mathcal{P}\mathcal{R}$-anti-slant warped product submanifold $M$ of the form $F\times_f N_{\lambda}$ if and only if shape operator of $M$ satisfies
\begin{align}\label{anti-shpcond}
A_{ntX}Z-A_{\varphi Z}tX = -\frac{1}{3}(\lambda)Z(\mu)X, \,\forall \, Z, W \in \Gamma(\mathfrak{D}^{\bot}\oplus\xi), X \in \Gamma(\mathfrak{D}_{\lambda}), 
\end{align}
for some function $\mu$ on $M$ such that $Y(\mu)=0$, $ Y \in \Gamma(\mathfrak{D}_{\lambda})$.
\end{theorem}
\begin{proof}
If $M$ is a $\mathcal{P}\mathcal{R}$-anti-slant warped product submanifold of a nearly paracosymplectic manifold $\widetilde{M}(\varphi,\xi,\eta,g)$. Then from formula-$(a)$ and formula-$(b)$ of lemma \ref{lem}, we derive Eq. \eqref{anti-shpcond}. Since $f$ is a function on $F$, setting $\mu = \ln f$  implies that $Y(\mu)=0$.
Conversely, let us assume that $M$ is $\mathcal{P}\mathcal{R}$-anti-slant submanifold of $\widetilde{M}(\varphi,\xi,\eta,g)$ such that Eq. \eqref{anti-shpcond} holds. Taking inner product of Eq. \eqref{anti-shpcond} with $W$ and from theorem \ref{pranti-thm2} we conclude that the integral manifold $F$ of $(\mathfrak{D}^{\bot}\oplus\xi)$ defines totally geodesic foliation in  $M$. Then by theorem \ref{pranti-thm1}, we find that the distribution $\mathfrak{D}_{\lambda}$ is integrable if and only if
\begin{align}\label{antiwp-thm1_0}
2\lambda g(\widetilde\nabla_{X}Y,Z)=g(A_{ntY}X,Z)+g(A_{ntX}Y,Z)-g(A_{\varphi Z}tY,X)-g(A_{\varphi Z}tX,Y).
\end{align}
From  Eq. \eqref{shp2form} and fact that $h$ is symmetric in  Eq. \eqref{antiwp-thm1_0}, we arrive at
\begin{align}\label{antiwp-thm1_1}
2\lambda g(\widetilde\nabla_{X}Y,Z)=g(A_{ntY}Z-A_{\varphi Z}tY,X)+g(A_{ntX}Z-A_{\varphi Z}tX,Y)
\end{align}
Eq. \eqref{anti-shpcond} yields by taking inner product with $Y$ that
\begin{align}\label{antiwp-thm1_11}
g(A_{ntX}Z-A_{\varphi Z}tX,Y)=-\frac{1}{3}g(\lambda Z(\mu)X,Y).
\end{align}
Now interchanging $X$ by $Y$ in Eq. \eqref{anti-shpcond} and taking inner product with $X$, we obtain that
\begin{align}\label{antiwp-thm1_111} 
g(A_{ntY}Z-A_{\varphi Z}tY,X)=-\frac{1}{3}g(\lambda Z(\mu)Y,X).
\end{align}
From Eqs. \eqref{antiwp-thm1_1}-\eqref{antiwp-thm1_111} and symmetry of $h$, we conclude that
\begin{align*}
g(h_{\lambda}(X, Y),Z) = -\frac{1}{3}g(Z(\mu)X, Y)=-\frac{1}{3}g(X, Y )g(\nabla{\mu}, Z).
\end{align*}
This implies $h_{\lambda}(X, Y)=-\frac{1}{3}g(X, Y )\nabla{\mu}$, where $h_{\lambda}$ is a second fundamental form of $\mathfrak{D}_{\lambda}$  in $M$ and  $\nabla\mu$ is gradient of $\mu = \ln{f}$. Hence, the integrable manifold of $\mathfrak{D}_{\lambda}$ is totally umbilical submanifold in $M$ and its mean curvature is non-zero and parallel and $Y(\mu)=0$ for all $Y \in \Gamma(\mathfrak{D}_{\lambda})$. Thus, by \cite{SH} we achieve that $M$ is a $\mathcal{P}\mathcal{R}$-anti-slant warped product submanifold of a nearly paracosymplectic manifold $\widetilde{M}(\varphi,\xi,\eta,g)$. This completes the proof of the theorem.  
\end{proof}
\acknowledgement{The authors are grateful to the anonymous reviewer for careful reading of the manuscript, constructive criticism and several valuable suggestions that improved the presentation of the work.}

\end{document}